\documentclass[12pt]{amsart}
\usepackage{amsmath,amssymb,amsfonts,amscd,epsfig}

\theoremstyle{plain}
\newtheorem{thm}{Theorem}
\newtheorem{lem}[thm]{Lemma}
\newtheorem{cor}[thm]{Corollary}
\newtheorem{prop}[thm]{Proposition}
\newtheorem{remark}[thm]{Remark}

\newtheorem{ex}[thm]{Example}
\newtheorem{method}[thm]{Method}
\newtheorem{result}[thm]{Result}
\numberwithin{thm}{section}
\numberwithin{equation}{section}

\newcommand{\A}{{\mathbb A}}

\newcommand{\FF}{{\mathbb F}}

\newcommand{\LL}{{\mathbb L}}

\newcommand{\PP}{{\mathbb P}}

\newcommand{\RR}{{\mathbb R}}

\newcommand{\ZZ}{{\mathbb Z}}

\title{Quantum Field Theory over $\FF_q$}
\author{Oliver Schnetz}\address{
Department Mathematik\\
Bismarkstra\ss e 1$\frac{1}{2}$\\
91054 Erlangen\\
Germany\\
E-mail address: schnetz@mi.uni-erlangen.de}
\begin{document}
\begin{abstract}
We consider the number $\bar N(q)$ of points in the projective complement of graph hypersurfaces over $\FF_q$ and
show that the smallest graphs with non-polynomial $\bar N(q)$ have 14 edges. We give six examples which fall
into two classes. One class has an exceptional prime 2 whereas in the other class $\bar N(q)$ depends on the number of cube roots of unity
in $\FF_q$. At graphs with 16 edges we find examples where $\bar N(q)$ is given by a polynomial in $q$ plus $q^2$ times the number of points
in the projective complement of a singular K3 in $\PP^3$.

In the second part of the paper we show that applying momentum space Feynman-rules over $\FF_q$ lets the perturbation series terminate for renormalizable
and non-re\-nor\-mal\-iz\-able bosonic quantum field theories.
\end{abstract}
\maketitle
\tableofcontents

\section{Introduction}

Inspired by the appearance of multiple zeta values in quantum field theories \cite{BK}, \cite{CENSUS} Kontsevich informally conjectured in 1997 that for every graph the number of zeros of the
graph polynomial (see Sect.\ \ref{s21} for a definition) over a finite field $\FF_q$ is a polynomial in $q$ \cite{KONT}. This conjecture puzzled graph theorists for quite a while.
In 1998 Stanley proved that a dual version of the conjecture holds for complete as well as for `nearly complete' graphs \cite{STAN}. The result was extended in 2000 by
Chung and Yang \cite{CY}. On the other hand, in 1998 Stembridge verified the conjecture by the Maple-implementation of a reduction algorithm for all graphs
with at most 12 edges \cite{STEM}. However, in 2000 Belkale and Brosnan were able to disprove the conjecture (in fact the conjecture is maximally false in a certain sense) \cite{BB}.
Their proof was quite general in nature and in particular relied on graphs with an apex (a vertex connected to all other vertices). This is not compatible with physical
Feynman rules permitting only low vertex-degree (3 or 4). It was still a possibility that the conjecture holds true for `physical' graphs where it originated from.
Moreover, explicit counter-examples were not known.

We show that the first counter-examples to Kontsevich's conjecture are graphs with 14 edges (all graphs with $\leq13$ edges are of polynomial type).
Moreover, these graphs are `physical': Among all `primitive' graphs with 14 edges in $\phi^4$-theory we find six graphs for which the number $\bar N(q)$
of points in the projective complement of the graph hypersurface (the zero locus of the graph polynomial) is not a polynomial in $q$.

Five of the six counter-examples fall into one class that has a polynomial behavior $\bar N(q)=P_2(q)$ for $q=2^k$ and $\bar N(q)=P_{\neq 2}(q)$ for all $q\neq2^k$ with
$P_2\neq P_{\neq 2}$ (although the difference between the two polynomials is minimal [Eqs.\ (\ref{22a1}) -- (\ref{22b2})])\footnote{D. Doryn proved independently in \cite{DORY1} that one
of these graphs is a counter-example to Kontsevich's conjecture.}. Of particular interest are three of the five graphs because for these the physical period is conjectured to be
a weight 11 multiple zeta value [Eq.\ (\ref{23})].
The sixth counter-example is of a new kind. One obtains three mutually (slightly) different polynomials $\bar N(q)=P_i(q)$, $i=-1,0,1$ depending on the remainder of $q$ modulo 3 [Eq.\ (\ref{22c})].

At 14 edges the breaking of Kontsevich's conjecture by $\phi^4$-graphs is soft in the sense that after eliminating the exceptional prime 2 (in the first case) or
after a quadratic field extension by cube roots of unity (leading to $q=1$ mod 3) $\bar N(q)$ becomes a polynomial in $q$.

At 16 edges we find two new classes of counter-examples. One resembles what we have found at 14 edges but provides three different polynomials depending on the remainder of $q$ modulo 4
[Eq.\ (\ref{22d})].

The second class of counter-examples from graphs with 16 edges is of an entirely new type. A formula for $\bar N(q)$ can be given that entails a polynomial in $q$ plus $q^2$ times
the number of points in the complement of a surface in $\PP^3$, Eqs.\ (\ref{22e1}) -- (\ref{22e6}). (The surface has been identified as a singular K3. It is a Kummer surface
with respect to the elliptic curve $y^2+xy=x^3-x^2-2x-1$, corresponding to the weight 2 level 49 newform \cite{BS}.)
This implies that the motive of the graph hypersurface is of non-mixed-Tate type.
The result was found by computer algebra using Prop.\ \ref{prop1} and Thm.\ \ref{thm1} which are proved with geometrical
tools that lift to the Grothendieck ring of varieties $K_0($Var$_k)$. This allows us to state the result as a theorem in the Grothendieck ring:
The equivalence class of the graph hypersurface $X$ of graph Fig.\ 1(e) minus vertex 2 is given by the Lefschetz motive $\LL=[\A^1]$ and the class $[F]$ of
the singular degree 4 surface in $\PP^3$ given by the zero locus of the polynomial
\begin{equation*}
a^2b^2+a^2bc+a^2bd+a^2cd+ab^2c+abc^2+abcd+abd^2+ac^2d+acd^2+bc^2d+c^2d^2\!,
\end{equation*}
namely (Thm.\ \ref{thm2})
\begin{eqnarray*}
[X]&=&\LL^{14}+\LL^{13}+4\LL^{12}+16\LL^{11}-8\LL^{10}-106\LL^9+263\LL^8-336\LL^7\\
&&\quad+\,316\LL^6-199\LL^5+45\LL^4+19\LL^3+[F]\LL^2+\LL+1.
\end{eqnarray*}

Although Kontsevich's conjecture does not hold in general, for physical graphs there is still a remarkable connection between $\bar N(q)$ and the quantum field theory period, Eq.\ (\ref{1a}).
In particular, in the case that $\bar N(q)$ is a polynomial in $q$ (after excluding exceptional primes and finite field extensions) we are able to predict the weight
of the multiple zeta value from the $q^2$-coefficient of $\bar N$ (see Remark \ref{rem1}).
Likewise, a non mixed-Tate $\LL^2$-coefficient $[F]$ in the above equation could indicate that the (yet unknown) period of the corresponding graph is not a multiple zeta value.

In Sect.\ \ref{sect3} we make the attempt to define a perturbative quantum field theory over $\FF_q$.
We keep the algebraic structure of the Feynman-amplitudes, interpret the integrands as $\FF_q$-valued functions and replace integrals by sums over $\FF_q$.
We prove that this renders many amplitudes zero (Lemma \ref{lem4}). In bonsonic theories with momentum independent vertex-functions only superficially convergent amplitudes survive.
The perturbation series terminates for renormalizable and non-renormalizable quantum field theories.
Only super-renormalizable quantum field theories may provide infinite (formal) power series in the coupling.
\vskip1ex

\noindent{\it Acknowledgements.}
The author is grateful for very enlightening discussions with S. Bloch and F.C.S. Brown on the algebraic nature of the counter-examples.
The latter carefully read the manuscript and made many valuable suggestions.
More helpful comments are due to S. Rams, F. Knop and P. M\"uller. H. Frydrych provided the author by a C$++$ class that facilitated
the counting in $\FF_4$ and $\FF_8$. Last but not least the author is grateful to J.R. Stembridge for making his beautiful programs publicly available
and to have the support of the Erlanger RRZE Computing Cluster with its friendly and helpful staff.

\section{Kontsevich's Conjecture}
\subsection{Fundamental Definitions and Identities}\label{s21}

Let $\Gamma$ be a connected graph, possibly with multiple edges and self-loops (edges connecting to a single vertex). We use $n$ for the number of edges of $\Gamma$.

The graph polynomial is a sum over all spanning trees $T$. Each spanning tree contributes by the product of variables corresponding to edges not in $T$,
\begin{equation}\label{1}
\Psi_\Gamma(x)=\sum_{T\,\rm span.\,tree}\;\prod_{e\not \in T}x_e.
\end{equation}
The graph polynomial was introduced by Kirchhoff who considered electric currents in networks with batteries of voltage $V_e$ and resistance $x_e$ at
each edge $e$ \cite{KIR}. The current through any edge is a rational function in the $x_e$ and the $V_e$ with
the common denominator $\Psi_\Gamma(x)$. In a tree where no current can flow the graph polynomial is 1.

The graph polynomial is related by a Cremona transformation $x\mapsto x^{-1}:=(x_e^{-1})_e$ to a dual polynomial built from the edges in $T$,
\begin{equation}\label{2}
\bar \Psi_\Gamma(x)=\sum_{T\,\rm span.\,tree}\;\prod_{e \in T}x_e\;=\;\Psi_\Gamma(x^{-1}) \prod_ex_e.
\end{equation}
The polynomial $\bar \Psi$ is dual to $\Psi$ in a geometrical sense: If the graph $\Gamma$ has a planar
embedding then the graph polynomial of a dual graph is the dual polynomial of the original graph. 
Both polynomials are homogeneous and linear in their coordinates and we have
\begin{equation}\label{2a}
\Psi_\Gamma=\Psi_{\Gamma-1}x_1+\Psi_{\Gamma/1},\quad\bar\Psi_\Gamma=\Psi_{\Gamma/1}x_1+\Psi_{\Gamma-1},
\end{equation}
where $\Gamma-1$ means $\Gamma$ with edge 1 removed whereas $\Gamma/1$ is $\Gamma$ with edge 1 contracted (keeping double edges, the graph polynomial
of a disconnected graph is zero).
The degree of the graph polynomial equals the number $h_1$ of independent cycles in $\Gamma$ whereas $\deg(\bar\Psi)=n-h_1$.

In quantum field theory graph polynomials appear as denominators of period integrals
\begin{equation}\label{1a}
P_\Gamma=\int_0^\infty\cdots\int_0^\infty\frac{{\rm d}x_1\cdots{\rm d}x_{n-1}}{\Psi_\Gamma(x)^2|_{x_n=1}}
\end{equation}
for graphs with $n=2h_1$. The integral converges for graphs that are primitive for the Connes-Kreimer coproduct
which is a condition that can easily be checked for any given graph (see Lemma 5.1 and Prop.\ 5.2 of \cite{BEK}).
If the integral converges, the graph polynomial may be replaced by its dual due to a Cremona transformation.

The polynomials $\Psi$ and $\bar \Psi$ have very similar (dual) properties. To simplify notation we mainly restrict
ourself to the graph polynomial although for graphs with many edges its dual is more tractable and was hence used in \cite{BB}, \cite{CY}, \cite{STAN}, and \cite{STEM}.

The graph polynomial (and also $\bar\Psi$) has the following basic property
\begin{lem}[Stembridge]\label{lem1}
Let $\Psi(x)=ax_ex_{e'}+bx_e+cx_{e'}+d$ for some variables $x_e$, $x_{e'}$ and polynomials $a,b,c,d$, then
\begin{equation}\label{3}
ad-bc=-\Delta_{e,e'}^2
\end{equation}
for a homogeneous polynomial $\Delta_{e,e'}$ which is linear in its variables.
\end{lem}
\begin{proof}
For the dual polynomial this is Theorem 3.8 in \cite{STEM}\footnote{In the version of \cite{STEM} that is available on Stembridge's homepage the theorem has the number 2.8.}.
The result for $\Psi$ follows by a Cremona transformation, Eq.\ (\ref{2}).
\end{proof}
As a simple example we take $C_3$, the cycle with 3 edges.
\begin{ex}\label{ex1}
\begin{eqnarray*}
\Psi_{C_3}(x)&=&x_1+x_2+x_3,\quad \Delta_{1,2}=1,\\
\bar\Psi_{C_3}(x)&=&x_1x_2+x_1x_3+x_2x_3,\quad \Delta_{1,2}=x_3.
\end{eqnarray*}
The dual of $C_3$ is a triple edge with graph polynomial $\bar\Psi_{C_3}$ and dual polynomial $\Psi_{C_3}$.
\end{ex}

The zero locus of the graph polynomial defines an in general singular projective variety, the graph hypersurface $X_\Gamma\subset\PP^{n-1}$.
In this article we consider the projective space over the field $\FF_q$ with $q$ elements.
Counting the number of points on $X_\Gamma$ means counting the number $N(\Psi_\Gamma)$ of zeros of $\Psi_\Gamma$. In this paper we prefer to (equivalently) count the points in the complement of the graph hypersurface.

In general, if $f_1,\ldots,f_m$ are homogeneous polynomials in $\ZZ[x_1,\ldots,x_n]$ and $N(f_1,\ldots,f_m)_{\FF_q^n}$ is the number of their common zeros
in $\FF_q^n$ we obtain for the number of points in the projective complement of their zero locus
\begin{eqnarray}\label{5}
\bar N(f_1,\ldots,f_m)_{\PP\FF_q^{n-1}}&=&|\{x\in\PP\FF_q^{n-1}|\exists i:f_i(x)\neq 0\}|\nonumber\\
&=&\frac{q^n-N(f_1,\ldots,f_m)_{\FF_q^n}}{q-1}.
\end{eqnarray}
If $\bar N$ is a polynomial in $q$ so is $N$ (and vice versa). We drop the subscript $\PP\FF_q^{n-1}$ if the context is clear.

The duality between $\Psi$ and $\bar\Psi$ leads to the following Lemma (which we will not use in the following).
\begin{lem}[Stanley, Stembridge]\label{lem2}
The number of points in the complement of the graph hypersurface can be obtained from the dual surface of the graph and its minors. Namely,
\begin{equation}\label{9}
\bar N(\Psi_\Gamma)=\sum_{T,S}(-1)^{|S|}\bar N(\bar\Psi_{\Gamma/T-S})
\end{equation}
where $T\sqcup S\subset E$ is a partition of an edge subset into a tree $T$ and an arbitrary edge set $S$ and $\Gamma/T-S$ is the contraction of $T$ in $\Gamma-S$.
\end{lem}
\begin{proof}
The prove is given in \cite{STEM} (Prop.\ 4.1) following an idea of \cite{STAN}.
\end{proof}

Calculating $\bar N(\Psi_\Gamma)$ is straightforward for small graphs. Continuing Ex.\ \ref{ex1} we find that $\Psi_{C_3}$ has $q^2$ zeros in $\FF_q^3$ (defining a hyperplane).
Therefore $\bar N(\Psi_{C_3})=(q^3-q^2)/(q-1)=q^2$.
The same is true for $\bar\Psi_{C_3}$, but here the counting is slightly more difficult. A way to find the result is to observe that whenever $x_2+x_3\neq0$ we
can solve $\bar\Psi_{C_3}=0$ uniquely for $x_1$. This gives $q(q-1)$ zeros. If, on the other hand, $x_2+x_3=0$ we conclude that $x_2=-x_3=0$ while $x_1$ remains arbitrary.
This adds another $q$ solutions such that the total is $q^2$.

A generalization of this method was the main tool in \cite{STEM} only augmented by the inclusion-exclusion formula $N(fg)=N(f)+N(g)-N(f,g)$.
We follow \cite{STEM} and denote for a fixed polynomial $f_1=g_1x_1-g_0$ with $g_1,g_0\in\ZZ[x_2,\ldots ,x_n]$ and any polynomial
$h=h_kx_1^k+h_{k-1}x_1^{k-1}+\ldots+h_0$ with $h_i\in\ZZ[x_2,\ldots ,x_n]$ the resultant of $f_1$ with $h$ as
\begin{equation}
\bar h=h_kg_0^k+h_{k-1}g_0^{k-1}g_1+\ldots +h_0g_1^k\in\ZZ[x_2,\ldots ,x_n].
\end{equation}

\begin{prop}[Stembridge]\label{prop0}
With the above notation we have
\begin{eqnarray}\label{7a}
N(f_1,\ldots,f_m)_{\FF_q^n}&=&N(g_1,g_0,f_2,\ldots,f_m)_{\FF_q^n}+N(\bar{f}_2,\ldots,\bar{f}_m)_{\FF_q^{n-1}}\nonumber\\
&&-\;N(g_1,\bar{f}_2,\ldots,\bar{f}_m)_{\FF_q^{n-1}}.
\end{eqnarray}
\end{prop}
\begin{proof}
Prop.\ 2.3 in \cite{STEM}.
\end{proof}
We continue to follow Stembridge and simplify the last term in the above equation. For a polynomial $h$ as defined above we write
\begin{equation}
\hat h=\left\{
\begin{array}{cl}
h_kg_0&\hbox{if }k>0\\
h_0&\hbox{if }k=0.
\end{array}\right.
\end{equation}
With this notation we obtain (Remark 2.4 in \cite{STEM})
\begin{equation}\label{60}
N(g_1,\bar{f}_2,\ldots,\bar{f}_m)=N(g_1,\hat{f}_2,\ldots,\hat{f}_m).
\end{equation}
Now we translate the above identities to projective complements, use the notation $f_1,\ldots,f_m={\bf f}_{1...m}={\bf f}$, and
add a rescaling property.
\begin{prop}\label{prop1}
Using the above notations we have for homogeneous polynomials $f_1,\ldots,f_m$
\begin{enumerate}
\item
\begin{equation}\label{6}
\bar N(f_1f_2,{\bf f}_{3...m})=\bar N(f_1,{\bf f}_{3...m})+\bar N(f_2,{\bf f}_{3...m})-\bar N(f_1,f_2,{\bf f}_{3...m})|_{\PP\FF_q^{n-1}},
\end{equation}
\item
\begin{equation}\label{7}
\bar N({\bf f})=\bar N(g_1,g_0,{\bf f}_{2...m})_{\PP\FF_q^{n-1}}+\bar N(\bar{\bf f}_{2...m})_{\PP\FF_q^{n-2}}-\bar N(g_1,\hat{\bf f}_{2...m})_{\PP\FF_q^{n-2}}.
\end{equation}
\item
If, for $I\subset \{1,\ldots ,n\}$ and polynomials $g,h\in\ZZ[(x_j)_{j\not\in I}]$, a coordinate transformation (re\-scal\-ing) $x_i\mapsto x_ig/h$ for
$i\in I$ maps ${\bf f}$ to $\tilde {\bf f}g^k/h^\ell$ with (possibly non-homogeneous) polynomials $\tilde {\bf f}$ and integers $k,\ell$ then ($\tilde{\bf f}=(\tilde f_1,\ldots,\tilde f_m)$),
\begin{equation}\label{8a}
\bar N({\bf f})_{\FF_q^n}=\bar N(gh,{\bf f})_{\FF_q^n}+\bar N(\tilde{\bf f})_{\FF_q^n}-\bar N(gh,\tilde{\bf f})_{\FF_q^n}.
\end{equation}
\end{enumerate}
\end{prop}
\begin{proof}
Eq.\ (\ref{6}) is inclusion-exclusion, Eq.\ (\ref{7}) is Prop.\ \ref{prop0} together with Eq.\ (\ref{60}).
Equation (\ref{8a}) is another application of inclusion-exclusion: On $gh\neq0$ the rescaling gives an isomorphism between the varieties defined by $\bf f$ and $\tilde{\bf f}$.
Hence in $\FF_q^n$ we have $N({\bf f})=N(gh,{\bf f})+N(\tilde{\bf f}|_{gh\neq0})$ and $N(\tilde{\bf f}|_{gh\neq0})=N(\tilde{\bf f})-N(gh,\tilde{\bf f})$.
Translation to complements leads to the result.
\end{proof}
In practice, one first tries to eliminate variables using (1) and (2). If no more progress is possible one may try to proceed with (3) (see the proof of Thm.\ \ref{thm2}).
In this case it may be convenient to work with non-homogeneous polynomials in affine space. One can always swap back to projective space by
\begin{equation}\label{8b}
N({\bf f})_{\PP\FF_q^{n-1}}=N({\bf f}|_{x_1=0})_{\PP\FF_q^{n-2}}+N({\bf f}|_{x_1=1})_{\FF_q^{n-1}}.
\end{equation}
This equation is clear by geometry. Formally, it can be derived from Eq.\ (\ref{8a}) by
the transformation $x_i\mapsto x_ix_1$ for $i>1$ leading to $\tilde{\bf f}={\bf f}|_{x_1=1}$.

In the case of a single polynomial we obtain (Eq.\ (\ref{11}) is Lemma 3.2 in \cite{STEM}):
\begin{cor}\label{cor1}
Fix a variable $x_k$. Let $f=f_1 x_k+f_0$ be homogeneous, with $f_1,f_0\in\ZZ[x_1,\ldots ,\hat{x_k},\ldots ,x_n]$.
If $\deg(f)>1$ then
\begin{equation}\label{11}
\bar N(f)=q\bar N(f_1,f_0)_{\PP\FF_q^{n-2}}-\bar N(f_1)_{\PP\FF_q^{n-2}}.
\end{equation}
If $f$ is linear in all $x_k$ and $0<\deg(f)<n$ then $\bar N(f)\equiv0\mod q$.
\end{cor}
\begin{proof}
We use Eq.\ (\ref{7}) for $f_1=f$. Because $\deg(f)>1$ neither $f_1$ nor $f_0$ are constants $\neq0$ in the first term on the right hand side.
Hence, a point in the complement of $f_1=f_0=0$ in $\PP\FF_q^{n-1}$ has coordinates $x$ with $(x_2,\ldots,x_n)\neq0$.
Thus $(x_2:\ldots:x_n)$ are coordinates in $\PP\FF_q^{n-2}$ whereas $x_1$ may assume arbitrary values in $\FF_q$. The second term
in Eq.\ (\ref{7}) is absent for $m=1$ and we obtain Eq.\ (\ref{11}). Moreover, modulo $q$ we have $\bar N(f)=-\bar N(f_1)_{\PP\FF_q^{n-2}}$.
We may proceed until $f_1=g$ is linear yielding $\bar N(f)=\pm\bar N(g)_{\PP\FF_q^{n-\deg(f)}}=\pm q^{n-\deg(f)}\equiv0$ mod $q$, because $\deg(f)<n$.
\end{proof}

In the case of two polynomials $f_1,f_2$ we obtain (Eq.\ (\ref{12}) is Lemma 3.3 in \cite{STEM}):
\begin{cor}\label{cor2}
Fix a variable $x_k$. Let $f_1=f_{11} x_k+f_{10}$, $f_2=f_{21} x_k+f_{20}$ be homogeneous, with $f_{11},f_{10},f_{21},f_{20},\in\ZZ[x_1,\ldots,\hat{x_k},\ldots ,x_n]$.
If $\deg(f_1)>1$, $\deg(f_2)>1$ then
\begin{equation}\label{12}
\bar N(f_1,f_2)=q\bar N(f_{11},f_{10},f_{21},f_{20})+\bar N(f_{11}f_{20}-f_{10}f_{21})-\bar N(f_{11},f_{21})|_{\PP\FF_q^{n-2}}.
\end{equation}
If $f_1,f_2$ are linear in all their variables, $f_{11}f_{20}-f_{10}f_{21}=\pm\Delta^2$, $\Delta\in\ZZ[x_1,\ldots,\hat{x_k},\ldots ,x_n]$
for all choices of $x_k$, $0<\deg(f_1)$, $0<\deg(f_2)$, and $\deg(f_1f_2)<2n-1$ then $\bar N(f_1,f_2)\equiv0\mod q$.
\end{cor}
\begin{proof}
Double use of Eq.\ (\ref{7}) and Eq.\ (\ref{6}) lead to
\begin{eqnarray}\label{12a}
\bar N(f_1,f_2)&=&\bar N(f_{11},f_{10},f_{21},f_{20})_{\PP\FF_q^{n-1}}\nonumber\\
&&\quad+\,\bar N(f_{11}f_{20}-f_{10}f_{21})_{\PP\FF_q^{n-2}}-\bar N(f_{11},f_{21})_{\PP\FF_q^{n-2}}.
\end{eqnarray}
If $\deg(f_1)>1$, $\deg(f_2)>1$ we obtain Eq.\ (\ref{12}) in a way analogous to the proof of the previous corollary.

If $f_{11}f_{20}-f_{10}f_{21}=\pm\Delta^2$ and $\deg(f_1f_2)<2n-1$ then $\deg(\Delta)<n-1$ and the second term on the right hand side is 0 mod $q$ by
Cor.\ \ref{cor1}. We obtain $\bar N(f_1,f_2)\equiv-\bar N(f_{11},f_{21})_{\PP\FF_q^{n-2}}$ mod $q$.
Without restriction we may assume that $d_1=\deg(f_1)<d_2=\deg(f_2)$ and continue eliminating variables until $f_{11}\in\FF_q^\times$. In this situation
Eq.\ (\ref{12a}) leads to
\begin{equation}\label{12b}
\bar N(f_1,f_2)\equiv\pm[\bar N(1)_{\PP\FF_q^{n-d_1}}+\bar N(\Delta)_{\PP\FF_q^{n-d_1-1}}-\bar N(1)_{\PP\FF_q^{n-d_1-1}}]\mod q.
\end{equation}
Still $0<\deg(\Delta)=(d_2-d_1+1)/2<n-d_1$ such that the middle term vanishes modulo $q$. The first and the third term add up to $q^{n-d_1}\equiv0\mod q$
because $d_1<n-1$.
\end{proof}

We combine both corollaries with Lemma \ref{lem1} to prove that $q^2|\bar N(\Psi_\Gamma)$ for every simple\footnote{A graph is simple if it has no multiple edges or self-loops.}
graph $\Gamma$ (Eq.\ (\ref{13}) is equivalent to Thm.\ 3.4 in \cite{STEM})
\begin{cor}\label{cor3}
Let $f=f_{11} x_1x_2+f_{10} x_1+f_{01} x_2+f_{00}$ be homogeneous with $f_{11}$, $f_{10}$, $f_{01}$, $f_{00}\in\ZZ[x_3,\ldots,x_n]$.
If $\deg(f)>2$ and $f_{11}f_{00}-f_{10}f_{01}=-\Delta_{12}^2$, $\Delta_{12}\in\ZZ[x_3,\ldots,x_n]$ then
\begin{eqnarray}\label{13}
\bar N(f)&=&q^2\bar N(f_{11},f_{10},f_{01},f_{00})\nonumber\\
&&\quad+\,q[\bar N(\Delta_{12})-\bar N(f_{11},f_{01})-\bar N(f_{11},f_{10})]+\bar N(f_{11})|_{\PP\FF_q^{n-3}}.
\end{eqnarray}
If $f$ is linear in all its variables, if the statement of Lemma \ref{lem1} holds for $f$ and any choice of variables $x_e,x_{e'}$, and if $0<\deg(f)<n-1$
then $\bar N(f)\equiv0\mod q^2$. In particular $\bar N(\Psi_\Gamma)=0\mod q^2$ for every simple graph with $h_1>0$.\end{cor}
\begin{proof}
Eq.\ (\ref{13}) is a combination of Eqs.\ (\ref{11}) and (\ref{12}). The second statement is trivial for $\deg(f)=1$ and straightforward for $\deg(f)=2$ using Cors.\ \ref{cor1} and \ref{cor2}.
To show it for $\deg(f)>2$ we observe that modulo $q^2$ the second term on the right hand side of Eq.\ (\ref{13}) vanishes due to Cors.\ \ref{cor1} and \ref{cor2}.
We thus have $\bar N(f)\equiv\bar N(f_{11})_{\PP\FF_q^{n-3}}$ mod $q^2$ and by iteration we reduce the statement to $\deg(f)=2$.
Any simple non-tree graph fulfills the conditions of the corollary by Lemma \ref{lem1}.
\end{proof}

The main theorem of this subsection treats the case in which a simple graph with vertex-connectivity\footnote{The vertex-connectivity is the minimal number of vertices
that, when removed, split the graph.} $\geq 2$ has a vertex with 3 attached edges (a 3-valent vertex).
We label the edges of the 3-valent vertex by 1, 2, 3 and apply Lemma \ref{lem1} with $e=1$, $e'=2$. We will prove that
\begin{eqnarray}\label{14}
\Delta_{12}&=&\Psi_{\Gamma-12/3}x_3+\Delta\quad\hbox{with}\\
\Delta&=&\frac{\Psi_{\Gamma-1/23}+\Psi_{\Gamma-2/13}-\Psi_{\Gamma-3/12}}{2}\in\ZZ[x_4,\ldots,x_n].\label{14b}
\end{eqnarray}
Here $\Gamma-1/23$ means $\Gamma$ with edge $1$ removed and edges 2, 3 contracted. Note that $\Gamma-12/3$ is the graph $\Gamma$ after the removal of the 3-valent
vertex.

\begin{thm}\label{thm1}
Let $\Gamma$ be a simple graph with vertex-connectivity $\geq 2$. Then
\begin{eqnarray}\label{13c}
\bar N(\Psi_\Gamma)&=&q^{n-1}+O(q^{n-3}),\\
\bar N(\Psi_\Gamma)&\equiv&0\mod q^2.\label{13d}
\end{eqnarray}
If $\Gamma$ has a 3-valent vertex with attached edges $1$, $2$, $3$ then
\begin{eqnarray}
\bar N(\Psi_\Gamma)&=&q^3\bar N(\Psi_{\Gamma-12/3},\Psi_{\Gamma-1/23},\Psi_{\Gamma-2/13},\Psi_{\Gamma/123})\nonumber\\
&&-\,q^2\bar N(\Psi_{\Gamma-12/3},\Psi_{\Gamma-1/23},\Psi_{\Gamma-2/13})|_{\PP\FF_q^{n-4}}\label{13b}\\
&=&q\bar N(\Psi_{\Gamma/3})_{\PP\FF_q^{n-2}}+q\bar N(\Delta_{12})_{\PP\FF_q^{n-3}}-q^2\bar N(\Delta)_{\PP\FF_q^{n-4}}.\label{15}
\end{eqnarray}
In particular,
\begin{equation}\label{15a}
\bar N(\Psi_\Gamma)\equiv q\bar N(\Delta_{12})_{\PP\FF_q^{n-3}}\equiv q^2\bar N(\Psi_{\Gamma-12/3},\Delta)_{\PP\FF_q^{n-4}} \mod q^3.
\end{equation}
If, additionally, an edge $4$ forms a triangle with edges $2$, $3$ we have
\begin{equation}\label{14a}
\delta=\frac{\Psi_{\Gamma-12/34}+\Psi_{\Gamma-24/13}-\Psi_{\Gamma-34/12}}{2}\in\ZZ[x_5,\ldots,x_n]
\end{equation}
and
\begin{eqnarray}\label{15b}
\bar N(\Psi_\Gamma)&\!\!=&\!q(q-2)\bar N(\Psi_{\Gamma-2/3})|_{\PP\FF_q^{n-3}}\nonumber\\
&&\!+\,q(q-1)[\bar N(\Psi_{\Gamma-12/3})+\bar N(\Psi_{\Gamma-24/3})]+q^2\bar N(\Psi_{\Gamma-2/34})|_{\PP\FF_q^{n-4}}\\
&&\!+\,q^2[\bar N(\Psi_{\Gamma-124/3})+\bar N(\Psi_{\Gamma-12/34})\nonumber\\
&&\quad\,-\,\bar N(\Psi_{\Gamma-124/3},\delta)-\bar N(\Psi_{\Gamma-12/34},\delta)-(q-2)\bar N(\delta)]|_{\PP\FF_q^{n-5}}.\nonumber
\end{eqnarray}
\end{thm}
\begin{proof}
A graph polynomial is linear in all its variables. Hence, a non-trivial factorization provides a partition of the graph into disjoint edge-sets and every factor is the
graph polynomial on the corresponding subgraph. The subgraphs are joined by single vertices and thus the graph has vertex-connectivity one.
Therefore, vertex-connectivity $\geq2$ implies that $\Psi_\Gamma$ is irreducible. If $\Psi=\Psi_1x_1+\Psi_0$ then $\Psi_1\neq0$ and gcd$(\Psi_1,\Psi_0)=1$.
Thus, the vanishing loci of the ideals $\langle\Psi_1\rangle$ and $\langle\Psi_1,\Psi_0\rangle$ have codimension 1 and 2 in $\FF_q^{n-1}$, respectively.
The affine version of Eq.\ (\ref{11}) is\footnote{This argument was pointed out by a referee.}
$N(\Psi)=q^{n-1}+qN(\Psi_1,\Psi_0)_{\FF_q^{n-1}}-N(\Psi_1)_{\FF_q^{n-1}}$ which gives $N(\Psi)=q^{n-1}+O(q^{n-2})$. Translation to the projective complement
yields Eq.\ (\ref{13c}) while (\ref{13d}) is Cor.\ \ref{cor3}.

Every spanning tree has to reach the 3-valent vertex. Hence $\Psi_\Gamma$ cannot have a term proportional to $x_1x_2x_3$.
Similarly, the coefficients of $x_1x_2$, $x_1x_3$, and $x_2x_3$ have to be equal to the graph polynomial of $\Gamma-12/3$.
Hence $\Psi_\Gamma$ has the following shape
\begin{equation*}
\Psi_{\Gamma-12/3}(x_1x_2\!+\!x_1x_3\!+\!x_2x_3)+\Psi_{\Gamma-1/23}x_1+\Psi_{\Gamma-2/13}x_2+\Psi_{\Gamma-3/12}x_3+\Psi_{\Gamma/123}.
\end{equation*}
From this we obtain
\begin{equation*}
\Delta_{12}^2=(\Psi_{\Gamma-12/3}x_3+\Delta)^2-\Delta^2+\Psi_{\Gamma-1/23}\Psi_{\Gamma-2/13}-\Psi_{\Gamma-12/3}\Psi_{\Gamma/123},
\end{equation*}
with Eq.\ (\ref{14b}) for $\Delta$ and non-zero $\Psi_{\Gamma-12/3}$ (because $\Gamma$ has vertex-connectivity $\geq2$).
The left hand side of the above equation is a square by Lemma \ref{lem1} which leads to Eq.\ (\ref{14}) plus
\begin{equation}\label{18}
\Psi_{\Gamma-12/3}\Psi_{\Gamma/123}-\Psi_{\Gamma-1/23}\Psi_{\Gamma-2/13}=-\Delta^2
\end{equation}
(which is Eq.\ (\ref{3}) for $\Gamma/3$). This leads to
\begin{equation}\label{19}
\Psi_{\Gamma-1/23}\Psi_{\Gamma-2/13}\equiv\Delta^2 \mod \Psi_{\Gamma-12/3}.
\end{equation}
Substitution of Eq.\ (\ref{14b}) into 4-times Eq.\ (\ref{18}) leads to
\begin{equation}\label{20}
\Psi_{\Gamma-3/12}\equiv\Psi_{\Gamma-2/13} \mod \langle\Psi_{\Gamma-12/3},\Psi_{\Gamma-1/23}\rangle,
\end{equation}
where $\langle\Psi_{\Gamma-12/3},\Psi_{\Gamma-1/23}\rangle$ is the ideal generated by $\Psi_{\Gamma-12/3}$ and $\Psi_{\Gamma-1/23}$.

A straightforward calculation eliminating $x_1$, $x_2$, $x_3$ using Eq.\ (\ref{13}) and Prop.\ \ref{prop1}
(one may modify the Maple-program available on the homepage of J.R. Stembridge to do this) leads to
\begin{eqnarray*}
\bar N(\Psi_\Gamma)&=&q^3\bar N(\Psi_{\Gamma-12/3},\Psi_{\Gamma-1/23},\Psi_{\Gamma-2/13},\Psi_{\Gamma-3/12},\Psi_{\Gamma/123})\\
&&+\,q^2\big[-\bar N(\Psi_{\Gamma-12/3},\Psi_{\Gamma-1/23},\Psi_{\Gamma-2/13},\Psi_{\Gamma-3/12})\\
&&\quad\quad+\,\bar N(\Psi_{\Gamma-12/3},\Psi_{\Gamma-1/23},\Psi_{\Gamma-2/13})+\bar N(\Psi_{\Gamma-12/3},\Delta)\\
&&\quad\quad-\,\bar N(\Psi_{\Gamma-12/3},\Psi_{\Gamma-2/13})-\bar N(\Psi_{\Gamma-12/3},\Psi_{\Gamma-1/23})\big]\Big|_{\PP\FF_q^{n-4}}.
\end{eqnarray*}
From this equation one may drop $\Psi_{\Gamma-3/12}$ by Eq.\ (\ref{20}). Now, replacing $\Delta$ by $\Delta^2$ and Eq.\ (\ref{19}) with inclusion-exclusion (\ref{6})
proves Eq.\ (\ref{13b}). Alternatively, we may use Eqs.\ (\ref{11}) and (\ref{13}) together with Eq.\ (\ref{14}) to obtain Eq.\ (\ref{15}).
By Cor.\ \ref{cor3} we have $\bar N(\Psi_{\Gamma/3})\equiv\bar N(\Psi_{\Gamma-12/3})$ $\equiv0$ mod $q^2$ and by Cor.\ \ref{cor1} we have $\bar N(\Delta)\equiv0 $ mod $q$ which makes
Eq.\ (\ref{15a}) a consequence of Eqs.\ (\ref{11}) and (\ref{15}).

The claim in case of a triangle 2, 3, 4 follows in an analogous way from Eq.\ (\ref{13b}): With the identities
\begin{eqnarray*}
&&\Psi_{\Gamma-12/3}=\Psi_{\Gamma-124/3}x_4+\Psi_{\Gamma-12/34},\quad\Psi_{\Gamma-1/23}=\Psi_{\Gamma-12/34}x_4,\\
&&\Psi_{\Gamma-2/13}=\Psi_{\Gamma-24/13}x_4+\Psi_{\Gamma-2/134},\quad\Psi_{\Gamma/123}=\Psi_{\Gamma-2/134}x_4,
\end{eqnarray*}
which follow from the definition of the graph polynomial, we prove (\ref{14a}) and
\begin{equation*}
\Psi_{\Gamma-124/3}\Psi_{\Gamma-2/134}-\Psi_{\Gamma-12/34}\Psi_{\Gamma-24/13}=-\delta^2
\end{equation*}
from Eq.\ (\ref{18}). With Prop.\ \ref{prop1} we prove Eq.\ (\ref{15b}).
\end{proof}
A non-computer proof of Eq.\ (\ref{13b}) can be found in \cite{BS}.

Every primitive $\phi^4$-graph comes from deleting a vertex in a 4-regular graph. Hence, for these graphs Eqs.\ (\ref{13b}) -- (\ref{15a}) are always applicable.
In some cases a 3-valent vertex is attached to a triangle. Then it is best to apply Prop.\ \ref{prop1} to Eq.\ (\ref{15b}) although
this equation is somewhat lengthy (see Thm.\ \ref{thm2}).

Note that Eq.\ (\ref{15a}) gives quick access to $\bar N(\Psi_\Gamma)$ mod $q^3$. In particular, we have the following corollary.
\begin{cor}\label{cor4}
Let $\Gamma$ be a simple graph with $n$ edges and vertex-connectivity $\geq2$.
If $\Gamma$ has a 3-valent vertex and $2h_1(\Gamma)<n$ then $\bar N(\Psi_\Gamma)\equiv0\mod q^3$.
\end{cor}
\begin{proof}
We have $\deg(\Psi_{\Gamma-12/3})=h_1-2$ and $\deg(\Delta)=h_1-1$ in Eq.\ (\ref{15a}), hence $\deg(\Psi_{\Gamma-12/3})+\deg(\Delta)<n-3$.
By the Ax-Katz theorem \cite{AX}, \cite{KATZ} we obtain $N(\Psi_{\Gamma-12/3},\Delta)_{\FF_q^{n-3}}\equiv0\mod q$ such that the corollary follows from Eq.\ (\ref{5}).
\end{proof}

If $2h_1=n$ we are able to trace $\bar N$ mod $q^3$ by following a single term in the reduction algorithm (for details see \cite{BS}):
Because in the rightmost term of Eq.\ (\ref{15a}) the sum over the degrees equals
the number of variables we can apply Eq.\ (\ref{12}) while keeping only the middle term on the right hand side. Modulo $q$ the first term vanishes
trivially whereas the third term vanishes due to the Ax-Katz theorem. As long as $f_{11}f_{20}-f_{10}f_{21}$ factorizes we can continue using Eq.\ (\ref{12}) which
leads to the `denominator reduction' method in \cite{BROW}, \cite{BY} with the result given in Eq.\ (\ref{20a}).

In the next subsection we will see that $\bar N(\Psi_\Gamma)$ mod $q^3$ starts to become non-polynomial for graphs with 14 edges (and $2h_1=n$) whereas higher powers
of $q$ stay polynomial (see Result \ref{res3}). On the other hand $\bar N$ mod $q^3$ is of interest in quantum field theory. It gives access to the most singular part
of the graph polynomial delivering the maximum weight periods and we expect the (relative) period Eq.\ (\ref{1a}) amongst those.
Moreover, $\Delta_{12}^2$ [as in Eq.\ (\ref{15a})] is the denominator of the integrand after integrating over $x_1$ and $x_2$ \cite{BROW}.

For graphs that originate from $\phi^4$-theory we make the following observations:
\begin{remark}[heuristic observations]\label{rem1}
Let $\Gamma$ be a 4-regular graph minus one vertex, such that the integral Eq.\ (\ref{1a}) converges.
Let $c_2(f,q)\equiv\bar N(f)/q^2\mod q$ for $f$ the graph polynomial $\Psi_\Gamma$ or its dual $\bar\Psi_\Gamma$. We make the following heuristic observations:
\begin{enumerate}
\item $c_2(\Psi_\Gamma,q)\equiv c_2(\bar\Psi_\Gamma,q)\mod q$.
\item If $\Gamma'$ is a graph with period $P_{\Gamma'}=P_\Gamma$ [Eq.\ (\ref{1a})] then $c_2(\Psi_\Gamma,q)\equiv c_2(\Psi_{\Gamma'},q)\mod q$.
\item If $c_2(\Psi_\Gamma,q)=c_2$ is constant in $q$ then $c_2=0$ or $-1$.
\item If $c_2(\Psi_\Gamma,p^k)$ becomes a constant $\tilde c_2$ after a finite-degree field extension and excluding a finite set of primes $p$ then $\tilde c_2=0$ or $\tilde c_2=-1$.
\item If $c_2=-1$ (even in the sense of {\rm (4)}) and if the period is a multiple zeta value then it has weight $n-3$, with $n$ the number of edges of $\Gamma$.
\item If $c_2=0$ and if the period is a multiple zeta value then it may mix weights. The maximum weight of the period is $\leq n-4$.
\item One has $c_2(\Psi_\Gamma,q)\equiv\bar N(\Delta_{e,e'})/q\mod q$ for any two edges $e,e'$ in $\Gamma$ (see Eq.\ (\ref{3}) for the definition of $\Delta_{e,e'}$).
An analogous equivalence holds for the dual graph polynomial $\bar\Psi_\Gamma$ which is found to give the same $c_2\mod q$ by observation {\rm (1)}.
\end{enumerate}
\end{remark}
We can only prove the first statement of (7).
\begin{proof}[Proof of the first statement of {\rm (7)}]
By the arguments in the paragraph following Cor.\ \ref{cor4} we can eliminate variables starting from $\bar N(\Delta_{e,e'})$ keeping
only one term mod $q^2$. In \cite{BROW} it is proved that one can always proceed until five variables (including $e,e'$) are eliminated
leading to the `5-invariant' of the graph. This 5-invariant is invariant under changing the order with respect to which the variables are eliminated.
This shows that $\bar N(\Delta_{e,e'})=\bar N(\Delta_{f,f'})$ mod $q^2$ for any four edges $e,e',f,f'$ in $\Gamma$. The equivalence in (7) follows from
Eq.\ (\ref{15a}) and the fact that $\Gamma$ has (four) 3-valent vertices.
\end{proof}
By the proven part of (7) we know that `denominator reduction' \cite{BROW} of a primitive graph $\Gamma$ gives $\bar N(\Gamma)$ mod $q^3$:
If a sequence of edges leads to a reduced denominator $\psi$ in $m$ (non-reduced) variables we have
\begin{eqnarray}\label{20a}
\bar N(\Psi)&\equiv&(-1)^m\bar N(\psi)_{\PP\FF_q^{m-1}},\hbox{ if }m\geq1,\\
\bar N(\Psi)&\equiv&-\bar N(\psi)\quad\quad\quad\quad\;,\hbox{ if }\psi\in\ZZ,\nonumber
\end{eqnarray}
where $\bar N(z)$ for $z\in\ZZ$ is 1 if gcd$(z,q)=1$ and 0 otherwise. This explains observations (3) and (4) for `denominator reducible' graphs (for which there
exists a sequence of edges, such that $\psi\in\ZZ$). In this situation observations (5) and (6) are proved in \cite{BROW}.
Moreover, for a class of not too complicated graphs (6) can be explained by means of \'etale cohomology and Lefschetz's fixed-point formula \cite{DORY}.

Of particular interest will be the case when $\bar N$ is a polynomial in $q$. In this situation we have the following statement.
\begin{lem}[Stanley]\label{lem3}
For homogeneous $f_1,\ldots,f_m$ let $\bar N(f_1,$ $\ldots,$ $f_m)_{\PP\FF_q^{n-1}}=c_0+c_1q+\ldots+c_{n-1}q^{n-1}$ be a polynomial in $q$.
We obtain for the local zeta-function $Z_q(t)$ of the projective zero locus $f_1=\ldots=f_m=0$,
\begin{equation}\label{10}
Z_q(t)=\prod_{k=0}^{n-1}(1-q^kt)^{c_k-1}.
\end{equation}
By rationality of $Z_q$ \cite{DWOR} we see that all coefficients $c_k$ are integers, hence $\bar N\in\ZZ[q]$.
\end{lem}
\begin{proof}
A straightforward calculation using Eq.\ (\ref{5}) shows that $Z_q(t)=\exp(\sum_{k=1}^\infty N_{\PP\FF_{q^k}^{n-1}}t^k/k)$ leads to Eq.\ (\ref{10}).
\end{proof}

We end this subsection with the following remark that will allows us to lift some results to general fields (see Thm.\ \ref{thm2}).
\begin{remark}\label{rem2}
All the results of this subsection are valid in the Gro\-then\-dieck ring of varieties over a field $k$ if $q$ is replaced by the equivalence class of the affine line $[\A_k^1]$.
\end{remark}
\begin{proof}
The results follow from inclusion-exclusion, Cartesian products, $\FF_q^\times$-fibrations which behave analogously in the Grothendieck ring.
\end{proof}

\subsection{Methods}

Our main method is Prop.\ \ref{prop1} applied to Thm.\ \ref{thm1}. Identities (1) and (2) of Prop.\ \ref{prop1} have been implemented by J.R. Stembridge in a nice
Maple worksheet which is available on his homepage. Stembridge's algorithm tries to partially eliminate variables and expand products in a balanced way
(not to generate too large expressions). But, actually, it turned out to be more efficient to completely eliminate variables and expand all products once the sequence of
variables is chosen in an efficient way. Thm.\ \ref{thm1} reflects this strategy by providing concise formulas for completely eliminating variables that are attached to a vertex
(and a triangle). A good sequence of variables will be a sequence that tries to complete vertices or cycles.
Such a sequence is related to \cite{BROW} by providing a small 'vertex-width'.
\begin{method}\label{met1}
Choose a sequence of edges $1,2,\ldots,n$ such that every sub-sequence $1,2,\ldots,k$ contains as many complete vertices and cycles as possible.
Start from Thm.\ \ref{thm1} (if possible). Pick the next variable in the sequence that can be eliminated completely (if any) and apply Prop.\ \ref{prop1} (2).
Factor all polynomials. Expand all products by Prop.\ \ref{prop1} (1).
Continue until no more variables can be eliminated completely (because no variable is linear in all polynomials).

Next, apply the above algorithm to each summand. Continue until Prop.\ \ref{prop1} (2) can no longer be applied (because no variable is linear in any polynomial).

Finally (if necessary), try to use Prop.\ \ref{prop1} (3) to modify a polynomial in such a way that it becomes linear in (at least) one variable.
If successful continue with the previous steps.
\end{method}
In most cases (depending on the chosen sequence of variables) graphs with up to 14 edges reduce completely and the above method provides a polynomial in $q$.
Occasionally one may have to stop the algorithm because it becomes too time-consuming. This depends on Maple's ability to factorize polynomials and to handle large expressions.

Working over finite fields we do not have to quit where the algorithm stops: We can still count for small $q$.
A side effect of the algorithm is that it eliminates many variables completely before it stops. This makes counting significantly faster.
If $\bar N$ is a polynomial, by Eqs.\ (\ref{13c}), (\ref{13d}) we have to determine the coefficients $c_2,c_3,\ldots,c_{n-3}$. We can do this for $n=14$ edges
by considering all prime powers $q\leq16$. By Lemma\ \ref{lem3} the coefficients have to be integers. Conversely, if interpolation does not provide
integer coefficients we know that $\bar N$ cannot be a polynomial in $q$. For graphs with 14 edges this is a time consuming though possible method even if hardly any variables
were eliminated. D. Doryn used a similar method to prove (independently) that one of the graphs obtained from deleting a vertex from Fig.\ 1(a) is a counter-example to Kontsevich's
conjecture \cite{DORY1}.

We implemented a more efficient polynomial-test that uses the heu\-ristic observation that the coefficients are not only integers but have small absolute value. This determines
the coefficients by the Chinese-Remainder-Theorem if $\bar N$ is known for a few small primes. For graphs with 14 edges it was sufficient to use $q=2$, 3, 5, and 7
because the coefficients are two-digit integers (we tested the results with $q=4$). For graphs with 16 edges we had additionally to count for $q=8$ and $q=11$.
\begin{method}\label{met2}
Select a set of small primes $p_1,p_2,\ldots,p_k$.
Evaluate $d_2(i)=\bar N(p_i)/p_i^2$ for these primes. Determine the smallest (by absolute value) common representatives $c_2$ of $d_2(i)\mod p_i$ (usually take the smallest one and maybe
the second smallest if it is not much larger than the smallest representative).
For each of the $c_2$ calculate $d_3(i)=(d_2(i)-c_2)/p_i$. Proceed as before to obtain a set of sequences $c_2$, $c_3$, $\ldots$, $c_{n-1}$.
If for one of the sequences one has $d_n(i)=0$ for all $i$ and [see Eq.\ (\ref{13c})] $c_{n-2}=0$, $c_{n-1}=1$ (and the set of sequences was not too large) then it is likely that
$\bar N(q)$ is a polynomial in $q$, namely $c_2q^2+c_3q^3+\ldots+c_{n-3}q^{n-3}+q^{n-1}\mod(q-p_1)(q-p_2)\cdots(q-p_k)$.

If $\bar N$ is a polynomial with coefficients $c_i$ such that $|c_i|<p_1p_2\cdots p_k/2$ then it is determined uniquely by the smallest representative for each $c_i$.
\end{method}

Note that one can use the above method to either test if $\bar N(q)$ is a polynomial in $q$ (this test may occasionally give a wrong answer in both directions if the set
of primes is taken too small) or to completely determine a polynomial $\bar N(q)$ with a sufficient number of primes taken into account. In any case, without a priory knowledge on the size
of the coefficients of $\bar N(q)$ the results gained with method \ref{met2} cannot be considered as mathematical truth in the strict sense.

Normally, one would use the smallest primes, but because (as we will see in the next subsection) $p=2$ may be an exceptional prime it is useful to try the method without $p=2$
if it fails when $p=2$ is included. Similarly one may choose certain subsets of primes (like $q=1$ mod 3) to identify a polynomial behavior after finite field extensions.

Because only few primes are needed to apply this method it can be used with no reduction beyond Thm.\ \ref{thm1} for graphs with up to 16 edges. Calculating modulo small
primes is fast in C$++$ and counting can easily be parallelized which makes this Method a quite practical tool.

The main problem is to find a result for $\bar N(q)$ if it is not a polynomial in $q$. It turned out that for $\phi^4$-graphs with 14 edges the deviation from being polynomial
can be completely determined mod $q^3$. This is no longer true for graphs with 16 edges, but at higher powers of $q$ we only find terms that we already had in graphs
with 14 edges (see Result \ref{res3}). Therefore a quick access to $\bar N(q)$ mod $q^3$ is very helpful.
\begin{method}\label{met3}
Determine $c_2(q)\equiv\bar N(q)/q^2 \mod q$ using Eq.\ (\ref{15a}) together with Eq.\ (\ref{12}) [or Eq.\ (\ref{20a})] and Remark \ref{rem1}.
Choose for each $q$ a representative $\tilde c_2(q)$ of $c_2(q)\mod q$. Check if $\bar N(q)/q^2-\tilde c_2(q)$ is a polynomial in $q$.
\end{method}
The result of this method obviously depends on the choice of the representatives $\tilde c_2(q)$. However, when we apply the method to examples in the next subsection we have
distinguished choices for $\tilde c_2(q)$ namely $\bar N(2)$, $\bar N(a^2+ab+b^2)$, $\bar N(a^2+b^2)$, and $\bar N(f)$ in Result \ref{res3}.

In practice it is often useful to combine the methods. Typically one would first run Method \ref{met1}. If it fails to deliver a complete reduction one may
apply Method \ref{met3} to determine its polynomial discrepancy and eventually Method \ref{met2} to determine the result.

\subsection{Results}
\begin{figure}[t]
\epsfig{file=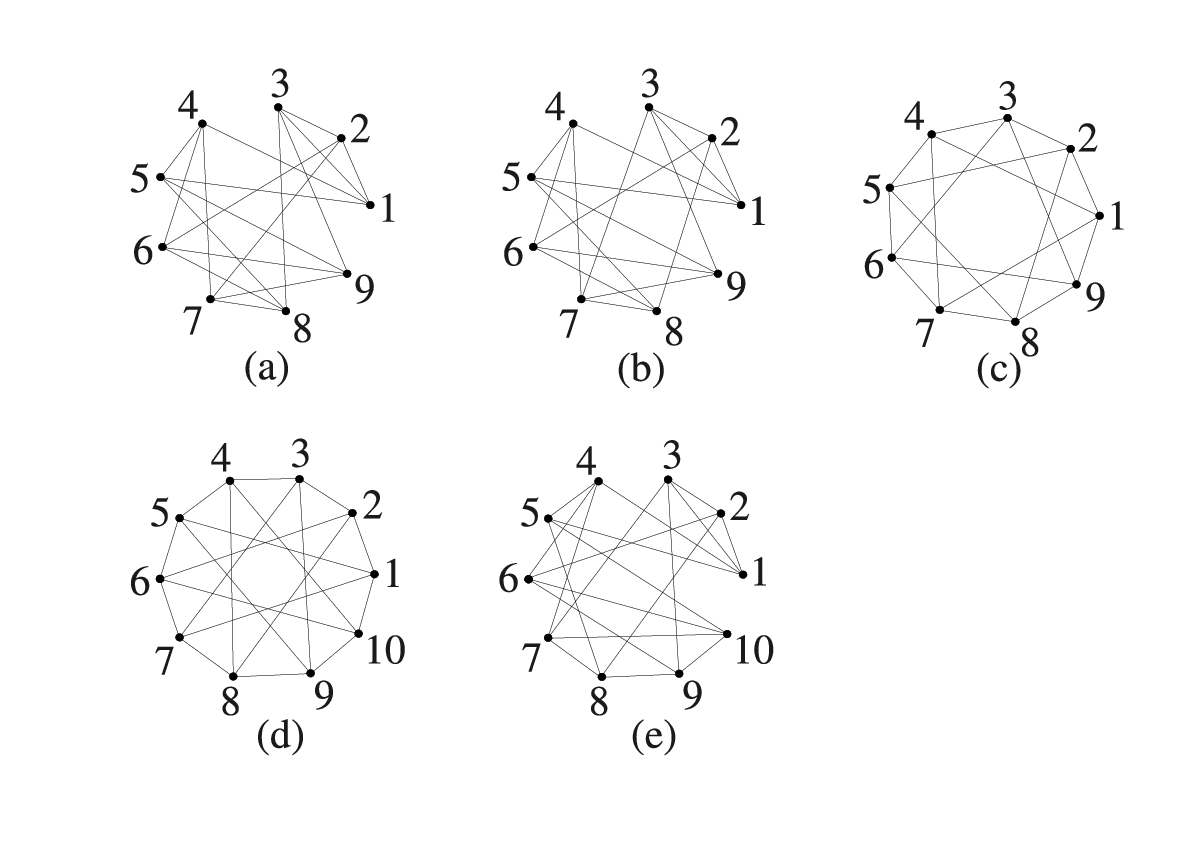,width=\textwidth}
\caption{4-regular graphs that deliver primitive $\phi^4$-graphs by the removal of a vertex. Every such $\phi^4$-graph is a counter-example to Kontsevich's conjecture.
Graphs (a) -- (c) give a total of six non-isomorphic counter-examples with 14 edges. Graphs (d), (e) provide another seven counter-examples with 16 edges.
The graph hypersurface of (e) minus any vertex entails a degree 4 non-mixed-Tate two-fold (a K3 \cite{BS}). The graphs are taken from \cite{CENSUS} where they have the names
$P_{7,8}$, $P_{7,9}$, $P_{7,11}$, $P_{8,40}$, and $P_{8,37}$, respectively. See Eqs.\ (\ref{22a1}) -- (\ref{22e6}) for the results.}
\end{figure}

First, we applied our methods to the complete list of graphs with 13 edges that are potential counter-examples to Kontsevich's conjecture. This list is due to the
1998 work by Stembridge and is available on his homepage. We found\footnote{We partly used Method \ref{met2} such that Result \ref{res1} should not be considered proven.}
that for all of these graphs $\bar N$ is a polynomial in $q$. This extends Stembridge's result \cite{STEM} from 12 to 13 edges.
\begin{result}\label{res1}
Kontsevich's conjecture holds for all graphs with $\leq13$ edges.
\end{result}

Second, we looked (using Method \ref{met2}) at all graphs with 14 edges that originate from primitive $\phi^4$-graphs [graphs with finite period (\ref{1a})].
These graphs come as 4-regular graphs with one vertex removed. They have $n=2h_1$ edges, 4 of which are 3-valent whereas all others are 4-valent.
A complete list of 4-regular graphs that lead to primitive $\phi^4$-graphs with up to 16 edges can be found in \cite{CENSUS}.
\begin{result}\label{res2}
Kontsevich's conjecture holds for all primitive $\phi^4$-graphs with 14 edges with the exception of the graphs obtained from Figs.\ 1(a) -- (c) by the removal of a vertex.
\end{result}

The counter-examples Fig.\ 1(a) -- (c) fall into two classes: One, Figs.\ 1(a), (b) with exceptional prime 2, second, Fig.\ 1(c) with a quadratic extension.
These counter-examples are the smallest counter-examples to Kontsevich's conjecture by Result \ref{res1}.

Next, we tested the power of our methods to primitive $\phi^4$-graphs with 16 edges. We scanned through the graphs with Method \ref{met3}
to see whether we find some new behavior. Only in the last five graphs of the list in \cite{CENSUS} we expect something new.
We were able to pin down the result for graphs coming from Fig.\ 1(d), (e). Figure 1(d) features a fourth root of unity extension together with an
exceptional prime 2 whereas Fig.\ 1(e) leads to a degree 4 surface in $\PP^3$ which is non-mixed-Tate.
\begin{result}\label{res3}
All graphs coming from Fig.\ 1 by the removal of a vertex are counter-examples to Kontsevich's conjecture (six with $14$ edges, seven with $16$ edges).
We list $\bar N(\Psi)/q^2$, the number of points in the projective complement of the graph hypersurface divided by $q^2$.
The second expression [in brackets] contains the result $\bar N(\bar\Psi)/q^2$ for the dual graph hypersurface.

In the following $\bar N(2)=\bar N(2)_{\PP\FF_q^0}=0$ if $q=2^k$ and $1$ otherwise,
$\bar N(a^2+ab+b^2)=\bar N(a^2+ab+b^2)_{\PP\FF_q^1}=q-\{1,0,-1\}$ if
$q\equiv1,0,-1\mod 3$, respectively,  $\bar N(a^2+b^2)=\bar N(a^2+b^2)_{\PP\FF_q^1}=q-\{1,0,-1\}$ if
$q\equiv1,0$ or $2,-1\mod 4$, respectively, and
\begin{eqnarray}\label{22}
f\;=\;f(a,b,c,d)&=&a^2b^2+a^2bc+a^2bd+a^2cd+ab^2c+abc^2\nonumber\\
&&\quad +\,abcd+abd^2+ac^2d+acd^2+bc^2d+c^2d^2.
\end{eqnarray}
\begin{eqnarray}\label{22a1}
&&\quad\quad\quad(1)\hbox{ Fig.\ 1(a) $-$ vertex $1$}\\
&&\hspace*{-14pt}q^{11}\!-\!q^8\!-\!24q^7\!+\!54q^6\!-\!36q^5\!-\!2q^4\!+\!34q^2\!-\!32q\!-\!\bar N(2)\nonumber\\
&&\hspace*{-14pt}[q^{11}\!-\!5q^8\!-\!11q^7\!+\!24q^6\!+\!q^5\!-\!50q^4\!+\!83q^3\!-\!47q^2\!-\bar N(2)]\nonumber
\end{eqnarray}
\begin{eqnarray}\label{22a2}
&&\quad\quad\quad(2)\hbox{ Fig.\ 1(a) $-$ vertex $2$, $3$, $4$, or $5$}\\
&&\hspace*{-14pt}q^{11}\!-\!3q^8\!-\!13q^7\!+\!34q^6\!-\!26q^5\!+\!13q^4\!-\!14q^3\!+\!13q^2\!-\!4q\!-\bar N(2)\nonumber\\
&&\hspace*{-14pt}[q^{11}\!-\!6q^8\!-\!6q^7\!+\!23q^6\!-\!9q^5\!-\!11q^4\!+\!10q^3\!+\!9q^2\!-\!12q\!-\bar N(2)]\nonumber
\end{eqnarray}
\begin{eqnarray}\label{22a3}
&&\quad\quad\quad(3)\hbox{ Fig.\ 1(a) $-$ vertex $6$, $7$, $8$, or $9$}\\
&&\hspace*{-14pt}q^{11}\!-\!4q^8\!-\!11q^7\!+\!38q^6\!-\!39q^5\!+\!24q^4\!-\!16q^3\!+\!11q^2\!-\!4q\!-\bar N(2)\nonumber\\
&&\hspace*{-14pt}[q^{11}\!-\!6q^8\!-\!6q^7\!+\!26q^6\!-\!12q^5\!-\!8q^4\!-\!7q^3\!+\!28q^2\!-\!16q\!-\bar N(2)]\nonumber
\end{eqnarray}
\begin{eqnarray}\label{22b1}
&&\quad\quad\quad(4)\hbox{ Fig.\ 1(b) $-$ vertex $1$, $2$, or $3$}\\
&&\hspace*{-14pt}q^{11}\!-\!3q^8\!-\!16q^7\!+\!41q^6\!-\!27q^5\!+\!q^4\!-\!5q^3\!+\!24q^2\!-\!18q\!-\bar N(2)\nonumber\\
&&\hspace*{-14pt}[q^{11}\!-\!5q^8\!-\!9q^7\!+\!28q^6\!-\!11q^5\!-\!10q^4\!+\!5q^3\!+\!13q^2\!-\!14q\!-\bar N(2)]\nonumber
\end{eqnarray}
\begin{eqnarray}\label{22b2}
&&\quad\quad\quad(5)\hbox{ Fig.\ 1(b) $-$ vertex $4$, $5$, $6$, $7$, $8$, or $9$}\\
&&\hspace*{-14pt}q^{11}\!-\!4q^8\!-\!13q^7\!+\!44q^6\!-\!46q^5\!+\!32q^4\!-\!29q^3\!+\!24q^2\!-\!9q\!-\bar N(2)\nonumber\\
&&\hspace*{-14pt}[q^{11}\!-\!5q^8\!-\!9q^7\!+\!34q^6\!-\!26q^5\!+\!5q^4\!-\!8q^3\!+\!18q^2\!-\!11q\!-\bar N(2)]\nonumber
\end{eqnarray}
\begin{eqnarray}\label{22c}
&&\quad\quad\quad(6)\hbox{ Fig.\ 1(c) $-$ any vertex}\\
&&\hspace*{-14pt}q^{11}\!-\!3q^8\!-\!15q^7\!+\!41q^6\!-\!32q^5\!+\!7q^4\!-\!3q^3\!+\!15q^2\!-\!15q\!+\bar N(a^2\!+\!ab\!+\!b^2)\nonumber\\
&&\hspace*{-14pt}[q^{11}\!-\!5q^8\!-\!9q^7\!+\!28q^6\!-\!7q^5\!-\!18q^4\!+\!3q^3\!+\!22q^2\!-\!17q\!+\bar N(a^2\!+\!ab\!+\!b^2)]\nonumber
\end{eqnarray}
\begin{eqnarray}\label{22d}
&&\quad\quad\quad(7)\hbox{ Fig.\ 1(d) $-$ any vertex}\\
&&\hspace*{-14pt}q^{13}\!-\!3q^{10}\!-\!11q^9\!+\!2q^8\!+\!90q^7\!-\!191q^6\!+\!208q^5\!-\!153q^4\!+\!79q^3\nonumber\\
&&\quad-[25+\bar N(2)]q^2\!-\!q\!+\!\bar N(a^2+b^2)\nonumber\\
&&\hspace*{-14pt}[q^{13}\!-\!7q^{10}\!-\!5q^9\!+\!9q^8\!+\!46q^7\!-\!108q^6\!+\!197q^5\!-\!294q^4\!+\!253q^3\nonumber\\
&&\quad-[105\!+\!\bar N(2)]q^2\!-\![q\!+\!8\bar N(2)]q\!+\!\bar N(a^2+b^2)]\nonumber
\end{eqnarray}
\begin{eqnarray}\label{22e1}
&&\quad\quad\quad(8)\hbox{ Fig.\ 1(e) $-$ vertex $1$}\\
&&\hspace*{-14pt}q^{13}\!-\!2q^{10}\!-\!19q^9\!+\!14q^8\!+\!103q^7\!-\!266q^6\!+\!374q^5\!-\!410q^4\!+\!322q^3\nonumber\\
&&\quad-97q^2\!-\!43q\!+\!\bar N(f)_{\PP\FF_q^3}\nonumber\\
&&\hspace*{-14pt}[q^{13}\!-\!5q^{10}\!-\!11q^9\!+\!8q^8\!+\!84q^7\!-\!187q^6\!+\!267q^5\!-\!386q^4\!+\!427q^3\nonumber\\
&&\quad-221q^2\!-\![11\!-\!2\bar N(a^2\!+\!ab\!+\!b^2)]q\!+\!\bar N(f)_{\PP\FF_q^3}]\nonumber
\end{eqnarray}
\begin{eqnarray}\label{22e2}
&&\quad\quad\quad(9)\hbox{ Fig.\ 1(e) $-$ vertex $2$ or $4$}\\
&&\hspace*{-14pt}q^{13}\!-\!3q^{10}\!-\!15q^9\!+\!9q^8\!+\!107q^7\!-\!262q^6\!+\!337q^5\!-\!315q^4\!+\!199q^3\nonumber\\
&&\quad-45q^2\!-\!19q\!+\!\bar N(f)_{\PP\FF_q^3}\nonumber\\
&&\hspace*{-14pt}[q^{13}\!-\!5q^{10}\!-\!12q^9\!+\!19q^8\!+\!63q^7\!-\!174q^6\!+\!229q^5\!-\!241q^4\!+\!181q^3\nonumber\\
&&\quad-50q^2\!-\![20\!-\!\bar N(a^2\!+\!ab\!+\!b^2)]q\!+\!\bar N(f)_{\PP\FF_q^3}]\nonumber
\end{eqnarray}
\begin{eqnarray}\label{22e3}
&&\quad\quad\quad(10)\hbox{ Fig.\ 1(e) $-$ vertex $3$ or $5$}\\
&&\hspace*{-14pt}q^{13}\!-\!3q^{10}\!-\!18q^9\!+\!25q^8\!+\!71q^7\!-\!214q^6\!+\!282q^5\!-\!246q^4\!+\!133q^3\nonumber\\
&&\quad-13q^2\!-\!24q\!+\!\bar N(f)_{\PP\FF_q^3}\nonumber\\
&&\hspace*{-14pt}[q^{13}\!-\!5q^{10}\!-\!13q^9\!+\!24q^8\!+\!56q^7\!-\!177q^6\!+\!255q^5\!-\!283q^4\!+\!212q^3\nonumber\\
&&\quad-54q^2\!-\!22q\!+\!\bar N(f)_{\PP\FF_q^3}]\nonumber
\end{eqnarray}
\begin{eqnarray}\label{22e4}
&&\quad\quad\quad(11)\hbox{ Fig.\ 1(e) $-$ vertex $6$}\\
&&\hspace*{-14pt}q^{13}\!-\!3q^{10}\!-\!21q^9\!+\!41q^8\!+\!36q^7\!-\!168q^6\!+\!237q^5\!-\!208q^4\!+\!93q^3\nonumber\\
&&\quad+24q^2\!-\!37q\!+\!\bar N(f)_{\PP\FF_q^3}\nonumber\\
&&\hspace*{-14pt}[q^{13}\!-\!5q^{10}\!-\!14q^9\!+\!27q^8\!+\!48q^7\!-\!161q^6\!+\!215q^5\!-\!199q^4\!+\!115q^3\nonumber\\
&&\quad-3q^2\!-\![29\!+\!2\bar N(2)]q\!+\!\bar N(f)_{\PP\FF_q^3}]\nonumber
\end{eqnarray}
\begin{eqnarray}\label{22e5}
&&\quad\quad\quad(12)\hbox{ Fig.\ 1(e) $-$ vertex $7$ or $8$}\\
&&\hspace*{-14pt}q^{13}\!-\!4q^{10}\!-\!16q^9\!+\!33q^8\!+\!38q^7\!-\!157q^6\!+\!214q^5\!-\!185q^4\!+\!96q^3\nonumber\\
&&\quad-7q^2\!-\!15q\!+\!\bar N(f)_{\PP\FF_q^3}\nonumber\\
&&\hspace*{-14pt}[q^{13}\!-\!5q^{10}\!-\!14q^9\!+\!32q^8\!+\!42q^7\!-\!170q^6\!+\!234q^5\!-\!200q^4\!+\!91q^3\nonumber\\
&&\quad+10q^2\!-\!22q\!+\!\bar N(f)_{\PP\FF_q^3}]\nonumber
\end{eqnarray}
\begin{eqnarray}\label{22e6}
&&\quad\quad\quad(13)\hbox{ Fig.\ 1(e) $-$ vertex $9$ or $10$}\\
&&\hspace*{-14pt}q^{13}\!-\!3q^{10}\!-\!15q^9\!+\!11q^8\!+\!99q^7\!-\!252q^6\!+\!333q^5\!-\!318q^4\!+\!213q^3\nonumber\\
&&\quad-61q^2\!-\!18q\!+\!\bar N(f)_{\PP\FF_q^3}\nonumber\\
&&\hspace*{-14pt}[q^{13}\!-\!5q^{10}\!-\!11q^9\!+\!13q^8\!+\!81q^7\!-\!210q^6\!+\!290q^5\!-\!329q^4\!+\!269q^3\nonumber\\
&&\quad-90q^2\!-\![24+2\bar N(2)]q+\bar N(f)_{\PP\FF_q^3}]\nonumber
\end{eqnarray}
\end{result}

Interestingly, the period Eq.\ (\ref{1a}) associated to Fig.\ 1(a), Eqs.\ (\ref{22a1}) -- (\ref{22a3}), has been determined by `exact numerical methods' as weight 11 multiple zeta value \cite{CENSUS}, namely
\begin{eqnarray}\label{23}
P_{7,8}&=&\frac{22383}{20}\zeta(11)-\frac{4572}{5}[\zeta(3)\zeta(5,3)-\zeta(3,5,3)]-700\zeta(3)^2\zeta(5)\nonumber\\
&&\quad+\,1792\zeta(3)\left[\frac{27}{80}\zeta(5,3)+\frac{45}{64}\zeta(5)\zeta(3)-\frac{261}{320}\zeta(8)\right],
\end{eqnarray}
where $\zeta(5,3)=\sum_{i>j}i^{-5}j^{-3}$ and $\zeta(3,5,3)=\sum_{i>j>k}i^{-3}j^{-5}k^{-3}$. So, a multiple zeta period does not imply
that $\bar N$ is a polynomial in $q$. The converse may still be true: If $\bar N$ is a polynomial in $q$ then the period (\ref{1a}) is a multiple zeta value.
It would be interesting to confirm that the period of Fig.\ 1(e) is not a multiple zeta value, but regretfully this is beyond the power of the present `exact numerical
methods' used in \cite{BK} and \cite{CENSUS}.

Most of the above results were found applying Method \ref{met2} at some stage. They are therefore not mathematically proven. However, due to numerous cross-checks
the author considers them as very likely true. We mainly worked with the prime-powers
$q=2$, 3, 4, 5, 7, 8, and 11. The counting for $q=8$ and $q=11$ for graphs with 16 edges (using Eqs.\ (\ref{13b}), (\ref{15b}) or analogous equations for the dual graph polynomial)
were performed on the Erlanger RRZE Computing Cluster.

Resorting to the counting Method \ref{met2} is not necessary for most graphs with 14 edges. Eqs.\ (\ref{15}) and  (\ref{15b}) of Thm.\ \ref{thm1} are powerful enough
to determine the results by pure computer-algebra. But in some cases finding good sequences can be time consuming and the 14-edge results had been
found by the author prior to Eqs.\ (\ref{15}) and  (\ref{15b}). The results have been checked by pure computer-algebra for Fig.\ 1(a) minus vertex 2, 3, 4, or 5
[Eq.\ (\ref{22a2})] and Fig.\ 1(e) minus vertex 2 or 4 [Eq.\ (\ref{22e2})].
In connection with Remark \ref{rem2} we can state the following theorem\footnote{A non-computer reduction of $c_2(q)$ to a singular K3 (isomorphic to $F$ in Thm.\ \ref{thm2})
for the graph Fig.\ 1(e) minus vertex 3 or 5 [see (\ref{22e3})] can be found in \cite{BS}.}:
\begin{thm}\label{thm2}
Let $\Gamma$ be the graph of Fig.\ 1(e) minus vertex $2$ (or minus vertex $4$) and $X$ its graph hypersurface in $\PP^{15}$ defined by the vanishing locus of graph polynomial $\Psi_\Gamma$.
Let $[X]$ be the image of $X$ in the Grothendieck ring $K_0($Var$_k)$ of varieties over a field $k$, let $\LL=[\A_k^1]$ be the equivalence class of the
affine line, and $1=[$Spec $k]$. With $[F]$ the image of the (singular) zero locus of $f$, given by Eq.\ (\ref{22}), in $\PP^3$ we obtain the identity
\begin{eqnarray}\label{23a}
[X]&=&\LL^{14}+\LL^{13}+4\LL^{12}+16\LL^{11}-8\LL^{10}-106\LL^9+263\LL^8-336\LL^7\nonumber\\
&&\quad+\,316\LL^6-199\LL^5+45\LL^4+19\LL^3+[F]\LL^2+\LL+1.
\end{eqnarray}
\end{thm}
\begin{proof}
By Remark \ref{rem2} and translation from complements to hypersurfaces in projective space Eq.\ (\ref{23a}) is equivalent to Eq.\ (\ref{22e2}).

To prove Eq.\ (\ref{22e2}) we use Eq.\ (\ref{15b}) in Thm.\ \ref{thm1} with edges 1, 2, 3, 4 corresponding to edges (1,3), (1,4), (1,5), (4,5) (edge (1,3) connects
vertex 1 with vertex 3 in Fig.\ 1(e), etc.). Terms without $\delta$ in Eq.\ (\ref{15b}) refer to minors of $\Gamma$. The most complicated of these is the first one
which has 14 edges and is isomorphic to Fig.\ 1(a) minus vertex 2. This minor has again a triangle with a 3-valent vertex such that Eq.\ (\ref{15b})
applies to it. Having two edges less than $\Gamma$ it is relatively easy to calculate $\bar N$ for this minor by Method \ref{met1} with the result given in Eq.\ (\ref{22a2})
[use e.g.\ the sequence  (1,3), (1,4), (1,5), (4,5), (3,9), (3,8), (5,8), (5,9), (4,6), (6,8), (7,8), (4,7), (6,9), (7,9)]. The other minors have 13 edges or less.
They give polynomial contributions to $\bar N(\Psi_\Gamma)$ by Result \ref{res1} which are easy to determine.

The first of the 3 terms containing $\delta$ in Eq.\ (\ref{15b}) can be reduced by Method \ref{met1} using the sequence
(4,7), (4,6), (3,7), (3,9), (6,9), (6,10), (9,10), (7,10), (7,8), (8,9), (5,8), (5,10). With the Maple 9.5-program used by the author (a modified version
of Stembridge's programs) it takes somewhat less than a day on a single core to produce the result which is the polynomial $q^{11}+q^{10}-q^9-6q^8-7q^7+51q^6-95q^5+101q^4-59q^3+11q^2+4q$.

The third term with $\delta$ is much simpler and produces $q^{11}-2q^9-10q^8+28q^7-25q^6+13q^5-18q^4+27q^3-16q^2-\bar N(2)q$ within two minutes using the sequence
(4,6), (6,9), (6,10), (9,10), (4,7), (3,9), (3,7), (5,10), (7,10), (7,8), (8,9), (5,8). Interestingly it cancels the $\bar N(2)$-dependence coming
from the 14-edge minor, Eq.\ (\ref{22a2}).

Only the second term with $\delta$ contains the degree 4 surface in $\PP^3$.
Eliminating variables according to the sequence (3,7), (3,9), (4,7), (4,6), (6,9), (6,10), (9,10), (5,10), (5,8), (8,9), (7,10), (7,8) (if possible) leaves us
(after about one day of computer algebra) with a degree 5 three-fold and two simpler terms which add to an expression
polynomial in $q$ after applying a rescaling, Eq.\ (\ref{8a}), to one of them.
The three-fold depends on the variables $x_{5,10}$, $x_{5,8}$, $x_{8,9}$, $x_{7,10}$, $x_{7,8}$ corresponding to the last five edges of the sequence.
To simplify the three-fold we first go to affine space using Eq.\ (\ref{8b}) with $x_1=x_{7,8}$. Afterwards we rescale $x_{5,10}$ and $x_{7,10}$ by the factor
$x_{5,8}x_{8,9}+x_{5,8}+x_{8,9}$ to obtain a degree 4 two-fold. We decided to apply another rescaling, namely $x_{7,10}\mapsto x_{7,10}(x_{8,9}+1)/x_{8,9}$, to
eliminate powers of 3 from the two-fold that otherwise would have appeared after going back to projective space using Eq.\ (\ref{8b}) backwards.
The variables $a$, $b$, $c$ in Eq.\ (\ref{22}) correspond to $x_{5,10}$, $x_{8,9}$, $x_{7,10}$, respectively. The variable $d$ is introduced by homogenizing the polynomial.
\end{proof}

Counting $\bar N(f)_{\PP\FF_p^3}$ mod $p$ for all primes $<10000$ we observe the following behavior: (This result is an immediate consequence of the fact that
$F$ is a Kummer surface \cite{BS}.)
\begin{result}\label{res4}
For $p>2$ we have $\bar N(f)_{\PP\FF_p^3}\equiv 28k(p)^2\mod p$ with $k(p)=0$ if $p=7$ or $p\equiv3$, $5$, $6$ mod $7$ ($-7$ is not a square in $\FF_p$) and
$k(p)\in \{1,2,\ldots,\lfloor \sqrt{p/7}\rfloor\}$ otherwise. We have (confirmed to 4 digits)
\begin{equation}\label{23b}
\sup_p \frac{7k(p)^2}{p}=1.
\end{equation}
\end{result}

Equation (\ref{23b}) gives us a hint that the surface $f=0$ cannot be reduced to a curve (or a finite field extension) because from the local zeta-function
and the Riemann hypothesis for finite fields we know \cite{HART} that the number of points on a projective non-singular curve of genus $g$ over $F_q$ is given by $q+1+\alpha$ with
$|\alpha|\leq 2g\sqrt{q}$. Thus, modulo $q$ this number is relatively close to 0 for large $q$. We cannot see such a behavior in Eq.\ (\ref{23b}).

We expect that the graphs derived from $P_{8,38}$, $P_{8,39}$, $P_{8,41}$ in \cite{CENSUS} also lead to 16-edge graphs which are counter-examples to Kontsevich's conjecture
none of which being expressible in terms of exceptional primes and finite field extensions. By an argument similar to the one above it seems that the
graph hypersurfaces of these graphs reduce to varieties of dimension $\geq2$. The (likely) absence of curves was not expected by the author.

\section{Outlook: Quantum Fields over $\FF_q$}\label{sect3}

In this section we try to take the title of the paper more literally. The fact that the integrands in Feynman-amplitudes are of algebraic nature allows us to make an attempt to define
a quantum field theory over a finite field $\FF_q$. Our definition will not have any direct physical interpretation. In particular, it should not
be understood as a kind of lattice regularization. In fact, the significance of this approach is unclear to the author.

We start from momentum space. The parametric space used in the previous section is not a good starting point
because it is derived from momentum or position space by an integral transformation that does not translate literally to finite fields.

We work in general space-time dimension $d$ and consider a bosonic quantum field theory with momentum independent vertex-functions. A typical candidate of such
a theory would be $\phi^k$-theory for any integer $k\geq 3$. In momentum space the `propagator' (see \cite{IZ}) is the inverse of a quadric in $d$ affine variables.
Normally one uses $Q=|p|^2+m^2$, where $|p|$ is the euclidean norm of $p\in\RR^d$ and $m$ is the mass of the particle involved.
One may use a Minkowskian metric (or any other metric) as well.

The denominator of the integrand in a Feynman amplitude is a product of $n$ quadrics $Q_i$ for a graph $\Gamma$ with $n$ (interior) edges.
The momenta in these propagators are sums or differences of $h_1$ momentum vectors, with $h_1$ the number of independent cycles of $\Gamma$.
The Feynman-amplitude of $\Gamma$ has the generic form
\begin{equation}\label{25}
A(\Gamma)=\int_{\RR^{dh_1}}{\rm d}^dp_1\cdots{\rm d}^dp_{h_1}\frac{1}{\prod_{i=1}^n Q_i(p)}.
\end{equation}
The asymptotical behavior of the differential form on the right hand side for large momenta is $\sim|p|^c$, where
\begin{equation}\label{26}
c=dh_1-2n
\end{equation}
is called the `superficial degree of divergence' (if $h_1>0$). It is clear that (at least) graphs with $c\geq0$ are ill-defined and need regularization.

From these amplitudes $A(\Gamma)$ we can construct a correlation function as sum over certain classes of graphs weighted by the order
of their automorphism groups,
\begin{equation}\label{27}
\Pi=\sum_\Gamma\frac{g^{|\Gamma|}A(\Gamma)}{|{\rm Aut}(\Gamma)|},
\end{equation}
where $g$ is the coupling and $|\Gamma|$ is an integer that grows with the size of $\Gamma$ (like $h_1$).
The correlation function demands renormalization to control the regularization of the single graphs.
For a renormalizable quantum field theory all graphs $\Gamma$ in the sum have the same superficial degree of divergence.
In a super-renormalizable theory (at low dimensions $d$) the divergence becomes less
for larger graphs, whereas the converse is true for a non-renormalizable theory (like quantum gravity).

Working over a finite field it seems natural to replace the integral in Eq.\ (\ref{25}) by a sum
\begin{equation}\label{28}
A(\Gamma)_{\FF_q}=\sum_{p\in\FF_q^{dh_1}\!:\,Q_i(p)\neq0}\frac{1}{\prod_{i=1}^n Q_i(p)}.
\end{equation}
The amplitude is well-defined (whereas $|{\rm Aut}(\Gamma)|$ in the denominator of Eq.\ (\ref{27}) causes problems for small $q$).
It is zero in many cases.
\begin{lem}\label{lem4}
Let $\Gamma$ be a graph with $n$ edges, $h_1>0$ independent cycles and superficial degree of divergence $c$. If $q>2$ then
\begin{equation}\label{29}
A(\Gamma)_{\FF_q}=0\quad\hbox{if}\quad(q-1)c+2n>0.
\end{equation}
\end{lem}
\begin{proof}
For all $x\in\FF_q^\times$ we have $x^{q-1}=1$. Hence the amplitude (\ref{28}) can be written as
\begin{equation}\label{30}
A(\Gamma)_{\FF_q}=\sum_{p\in\FF_q^{dh_1}}\prod_{i=1}^n Q_i(p)^{q-2}
\end{equation}
where the restriction to non-zero $Q_i$ can be dropped for $q>2$. The right hand side is a polynomial in the coordinates of the $p_i$ of degree $2n(q-2)$.
On the other hand we have (we use $0^0:=1$)
\begin{equation}\label{31}
\sum_{x\in\FF_q}x^k=\left\{\begin{array}{rl}-1&\hbox{if }0<k\equiv0\mod(q-1)\\
0&\hbox{otherwise,}\end{array}\right.
\end{equation}
which is obvious if one multiplies both sides of the equation by any $1\neq y^k\in\FF_q^\times$ [if $k\not\equiv0$ mod $(q-1)$].
In particular, the sum over a polynomial in $x$ vanishes unless the
polynomial has a minimum degree $q-1$. In case of $dh_1$ variables we need a minimum degree $dh_1(q-1)$.
The right hand side of (\ref{30}) does not have this minimum degree if $2n(q-2)<dh_1(q-1)$ which by Eq.\ (\ref{26}) gives Eq.\ (\ref{29}).
\end{proof}
We see that only superficially convergent graphs (with $c<0$) can give a non-zero amplitude. The complexity of the graph is limited by $q-1$ times
the degree of convergence. This means for the three possible scenarios of quantum field theory:
\begin{enumerate}
\item If the quantum field theory is non-renormalizable then $c$ becomes positive for sufficiently large graphs.
All correlation functions are polynomials in the coupling $g$ of universal ($q$-independent) maximum degree.
\item If the quantum field theory is renormalizable then $c$ is constant for all graphs that contribute to a correlation function.
The correlation function becomes a polynomial in the coupling with degree that may grow with $q$.
If the correlation function has $c\geq0$ only the tree level (with $h_1=0$) contributes.
\item If the quantum field theory is super-renormalizable then $c$ becomes negative for sufficiently large graphs.
In this case all correlation functions may be infinite (formal) power series.
\end{enumerate}
It is interesting to observe that finite fields give an upside down picture to normal quantum field theories. The most problematic non-renormalizable
quantum field theories give the simplest results whereas the most accessible super-renormalizable theories may turn out to be the most complicated ones
over finite fields. In between we have the renormalizable quantum field theories that govern the real world.

Another theme of interest could be an analogous study of $p$-adic quantum field theories.
\bibliographystyle{plain}
\renewcommand\refname{References}

\end{document}